\documentclass{amsart}

\usepackage{amsfonts,amsmath,amsthm,amssymb,color,stmaryrd}
\usepackage[all]{xy}
\usepackage[pagebackref, colorlinks=true, linkcolor=blue, citecolor=blue]{hyperref}

\numberwithin{equation}{section} %numerazione delle equazioni dentro le sezioni

\newtheorem{theorem}{Theorem}[section]
\newtheorem{lemma}[theorem]{Lemma}

\newtheorem{proposition}[theorem]{Proposition}

\newtheorem{corollary}[theorem]{Corollary}

\theoremstyle{definition}
\newtheorem{definition}[theorem]{Definition}
\newtheorem{example}[theorem]{Example}

\theoremstyle{remark}

\newcommand{\sC}{\mathcal{C}}% calligraphic math
\newcommand{\sE}{\mathcal{E}}% calligraphic math
\newcommand{\sF}{\mathcal{F}}% calligraphic math
\newcommand{\Oh}{\mathcal{O}}% calligraphic math
\newcommand{\sO}{\mathcal{O}}% calligraphic math
\newcommand{\sW}{\mathcal{W}}% calligraphic math
\newcommand{\K}{\mathbb{K}}% blackboard math
\newcommand{\Z}{\mathbb{Z}}% blackboard math
\newcommand{\C}{\mathbb{C}}% blackboard math
\newcommand{\debar}{\overline{\partial}}

\newcommand{\Hom}{\operatorname{Hom}}
\newcommand{\Ext}{\operatorname{Ext}}

\newcommand{\Aut}{\operatorname{Aut}}

\newcommand{\Span}{\operatorname{Span}}

\newcommand{\Id}{\operatorname{Id}}
\newcommand{\Tr}{\operatorname{Tr}}

\newcommand{\RHom}{R\mspace{-2mu}\operatorname{Hom}}
\newcommand{\sG}{\mathcal{G}}% calligraphic math
\newcommand{\sI}{\mathcal{I}}% calligraphic math
\newcommand{\sA}{\mathcal{A}}% calligraphic math

\newcommand{\HOM}{{{\mathcal H}om}}

\begin{document}

\title{Formality conjecture for minimal surfaces of Kodaira dimension 0}

\date{July 2, 2020}%ricordarsi di togliere il today per la versioni arXiv

\author[R. Bandiera]{Ruggero Bandiera}
\address{\newline
Universit\`a degli studi di Roma La Sapienza,\hfill\newline
Dipartimento di Matematica Guido
Castelnuovo,\hfill\newline
P.le Aldo Moro 5,
I-00185 Roma, Italy.}
\email{bandiera@mat.uniroma1.it}

\author[M. Manetti]{Marco Manetti}
\email{manetti@mat.uniroma1.it}
\urladdr{www.mat.uniroma1.it/people/manetti/}

\author[F. Meazzini]{Francesco Meazzini}
\email{meazzini@mat.uniroma1.it}
%\urladdr{www.mat.uniroma1.it/people/meazzini/}


\begin{abstract} Let $\sF$ be a polystable sheaf on
a smooth minimal projective surface of Kodaira dimension 0. Then the DG-Lie algebra 
$R\Hom(\sF,\sF)$ of derived endomorphisms of $\sF$ is formal. 
The proof is based on the study of equivariant $L_{\infty}$ minimal models of DG-Lie algebras equipped 
with a cyclic structure of degree 2 which is non-degenerate in cohomology, 
and does not rely (even for K3 surfaces) on 
previous results on the same subject. 
\end{abstract}

\maketitle

\section{Introduction}

The main goal of this paper is to provide an elementary proof of the following theorem, that extends an analogous result for K3 surfaces \cite{BZ18}.

\begin{theorem}[=Theorem~\ref{thm.formality2}]
Let $X$ be a smooth minimal projective surface of Kodaira dimension $0$, and consider a polystable sheaf $\sF$ on $X$. Then the DG-Lie algebra $\RHom_X(\sF,\sF)$ is formal.
\end{theorem}

Moduli spaces of coherent sheaves on K3 surfaces and Abelian surfaces have been intensively studied during the last decades. Among the reasons behind the interest in these objects there is certainly the fact due to Mukai that the smooth locus of the moduli space inherits a holomorphic symplectic structure from the symplectic form on the surface, \cite{Muk84}. In particular, provided that such a moduli space is smooth and projective, it yields an example of irreducible holomorphic symplectic manifold. In general the moduli space is singular at a point corresponding to a strictly semistable sheaf; these singularities arise either when the Mukai vector is not primitive or when the polarization on the surface is not general (i.e. it lies on a wall with respect to the walls and chambers decomposition of the ample cone, \cite{KLS06,Yos01}). Nevertheless, in some cases there exist symplectic resolutions, which have been investigated for moduli spaces with general polarization and non-primitive Mukai vector: first O'Grady found out two new examples of irreducible holomorphic symplectic manifolds \cite{OG99,OG03} by exhibiting symplectic resolutions of moduli spaces of sheaves on a K3 surface and on an Abelian surface. Few years later Kaledin-Lehn-Sorger showed that other than the ones in O'Grady's examples such moduli spaces do not admit symplectic resolutions, \cite{KLS06}.

More recently, in \cite{AS} Arbarello-Sacc\`a turned the attention to the case of a K3 surface with a non-general polarization and Mukai vector $(0, c_1, \chi)$. The corresponding moduli space admits a symplectic resolution, given by moving the polarization (hence changing the notion of stability) into a chamber, and they give a local description of the moduli space around the singularity in terms of a suitable Nakajima quiver variety.

By general deformation theory, an easy description of an analytic neighborhood around a singular point $[\sF]$ in the moduli space corresponding to a given (possibly non-general) polarization can be deduced from the formality of the derived endomorphisms of the sheaf $\sF$ on the surface $X$. We now briefly recall the main steps that led to the so called Kaledin-Lehn \emph{formality conjecture}. It is well known that  the base space of the formal semiuniversal deformation  of $[\sF]$ is the scheme-theoretic fibre of the Kuranishi map
\[ k\colon \widehat{\Ext_X^1(\sF,\sF)} \to \Ext_X^2(\sF,\sF)_0 = \ker(\Tr\colon \Ext_X^2(\sF,\sF)\to H^2(X,\sO_X)\cong \C) \]
which can be chosen to be $G$-equivariant (see e.g.~\cite{AS,quadraticity,Rim80}) with respect to the action of the automorphisms group modulo the action of the scalars: $G=\Aut(\sF)/\C^{\ast}$. 
 Often it is definitely not trivial to compute the null-fiber of the Kuranishi map; on the other hand its quadratic part $k_2$ is nothing but the Yoneda pairing, so that in general it is much easier to understand $k_2^{-1}(0)$ instead of $k^{-1}(0)$. In \cite{KL07}, Kaledin-Lehn essentially conjectured that for a polystable sheaf on a K3 surface the Kuranishi map is \emph{quadratic}, namely $k_2^{-1}(0)\cong k^{-1}(0)$. If this condition is satisfied then the moduli space, locally around $[\sF]$, is isomorphic to the GIT quotient $k^{-1}_2(0)\!\sslash\! G$.

In their original paper Kaledin-Lehn gave a first example motivating and inspiring the future work on the subject. The conjecture has been then proven in full generality by Yoshioka \cite{Yos16}, and partially by Arbarello-Sacc\`a \cite{AS}. Let us point out few remarks before going on. First recall that to any (homotopy class of a) DG-Lie algebra it is associated a deformation functor (see e.g. \cite{Man09,LMDT}), which in turn provides a Kuranishi map via the Maurer-Cartan equation. Moreover, if the DG-Lie algebra $L$ is formal (i.e. it is quasi-isomorphic to its cohomology) then the associated Kuranishi space $k^{-1}(0)$ is the null-fiber of the cup product in cohomology $H^1(L)\to H^2(L)$. Hence, by showing the formality of the DG-Lie algebra $R\Hom(\sF,\sF)$ one also proves the quadraticity of the Kuranishi map.
Notice that since our approach involves techniques of $L_{\infty}$-algebras we investigate derived endomorphisms as a DG-Lie algebra, while the papers~\cite{KL07,BZ18} consider $R\Hom(\sF,\sF)$ as an associative DG-algebra and also Kaledin's refinement of the Massey products works in the associative setting,~\cite{Kal07}. It is important to point out that the formality in the associative case is a stronger statement, but on the other hand the DG-Lie formality is the one needed for applications to moduli spaces.

It is worth mentioning that formality is in general much stronger and harder to prove than the quadraticity property; from the point of view of derived algebraic geometry this can be easily understood since formality implies that the derived moduli space is  locally quadratic. Nevertheless, formality of $R\Hom(\sF,\sF)$ has been conjectured for polystable sheaves again by Kaledin-Lehn in \cite{KL07}, it has been studied in some cases by Zhang \cite{Zh12}, and finally completely solved by Budur-Zhang \cite{BZ18} who proved that the conjecture holds true for any polystable sheaf using results about strong uniqueness of DG-enhancements.

It is interesting to notice that all of the above cited formality results actually rely on the famous result due to Kaledin about formality in families \cite{Kal07,Lun10}. Even if the vanishing of Massey products does not guarantee the formality of a DG-algebra $A$, see e.g. \cite{HS79}, Kaledin determined a refinement of them defining the so-called \emph{Kaledin class}  in a certain (reduced) Hochschild homology group depending on $A$. Furthermore, he proved that the variation of such a class in a suitable family of DG-algebras $\sA\to S$ over an irreducible base $S$ glue to a global section of a certain obstruction bundle $\sO b_S$ defined on $S$. It follows that if $\sO b_S$ does not admit non-trivial global sections then all the fibers $\sA_s$ are formal, \cite[Theorem 4.3]{Kal07}.

Applying Kaledin's result and twistor spaces, in the paper \cite{KL07} Kaledin-Lehn first obtained the formality of $\RHom(\sF,\sF)$ for sheaves of the form $\sF=\sI_Z^{\oplus n}$, where $\sI_Z$ denotes the ideal sheaf of some $0$-dimensional closed subscheme $Z$. Later, Zhang showed that Kaledin's theorem may be applied to polystable sheaves with some constraints on the ranks of the corresponding stable summands \cite[Proposition 1.3]{Zh12}, hence enlarging the class of polystable sheaves for which the formality conjecture holds. Eventually in \cite{BZ18} Budur-Zhang established a very interesting result, namely that the formality of derived endomorphisms of any object in $D^b(X)$ is preserved under derived equivalences; hence the formality conjecture follows since by \cite{Yos09} any polystable sheaf can be mapped via a Fourier-Mukai transform to another polystable sheaf satisfying the hypothesis of \cite[Proposition 3.1]{Zh12}.

In our recent paper \cite{quadraticity}, we proved that for a sheaf $\sF$ whose automorphisms group is reductive, e.g. for any $\sF$ polystable, the quadraticity of the Kuranishi map and the formality of the DG-Lie algebra $R\Hom(\sF,\sF)$ are in fact equivalent conditions; our proof has the avail of relaxing the hypothesis on the surface which no longer needs to be a K3. This provides another evidence of the formality conjecture without involving powerful methods of DG-category theory, but instead relying on the work of Yoshioka \cite{Yos16} and Arbarello-Sacc\`a \cite{AS}. Actually, both the papers \cite{AS,Yos16} base their proofs of the quadraticity property on the fundamental work \cite{Zh12}, hence again Kaledin's theorem \cite{Kal07} seems to be essential.

The present paper aims to prove the formality conjecture for polystable sheaves on a smooth minimal projective surface of Kodaira dimension $0$. Examples of such surfaces include projective K3 surfaces, Enriques surfaces, bielliptic surfaces and Abelian surfaces, \cite{BPV, Beauville}.
One of the main innovations of our proof is that we translate the problem into a purely algebraic statement (see Theorem~\ref{thm.formality}) about formality of DG-Lie algebras endowed with some additional structure (see Definition~\ref{def.qcyclicdgla}), which will be proved only involving elementary techniques of (strong homotopic) DG-Lie algebras. In particular, maybe surprisingly, 
in the case of K3 and Abelian surfaces our proof of the formality conjecture only requires a basic knowledge of $L_\infty$ algebras and is self-contained, meaning that it does not involve neither Kaledin's result about formality in families nor the geometric situations considered by Zhang in \cite{Zh12}.
As pointed out by one of the referees, it is similar in spirit to Neisendorfer and Miller's proof of the fact that any 6-dimensional simply connected Poincar\'e duality space is formal \cite{NM78}.

The plan of the paper is as follows.
In Section~\ref{sec.review} we fix notations and briefly summarise the results needed in the rest of the paper about formality and $L_{\infty}$ algebras. In Section \ref{section.qcyclicDGLA} we introduce the notion of quasi-cyclic DG-Lie algebras and discuss examples arising from geometric situations: a DG-Lie algebra $(L,d,[-,-])$ with finite-dimensional cohomology equipped with a degree $-n$ symmetric bilinear form $(-,-)\colon L^{\odot2}\to\K[-n]$ is called \emph{quasi-cyclic of degree} $n$ provided that 
\[ (dx,y)=(-1)^{|x|+1}(x,dy),\qquad ([x,y],z) = (x,[y,z]),\qquad\forall\,x,y,z\in L \]
and the form induced induced in cohomology $(-,-)\colon H(L)^{\odot2} \to\K[-n]$ is non-degenerate. 

The typical example of quasi-cyclic DG-Lie algebra of degree $n$ is given by the Dolbeault resolution
$L=A^{0,*}_X(\HOM(\sE,\sE))$ of the sheaf of endomorphisms of a locally free sheaf $\sE$
on an $n$-dimensional manifold $X$ equipped with a nowhere vanishing holomorphic volume form $\omega_X$, with the pairing  $(f,g)=\int_X \omega_X\wedge \Tr(fg)$, see Example~\ref{ex.VBonK3}. A similar construction can be  also performed when $\sE$ is replaced by any coherent sheaf, see Section~\ref{sec.formality}.

Then Section~\ref{sec.formalityforcyclicDGLA} is entirely devoted to the  proof of our main algebraic result.
	
\begin{theorem}[=Theorem \ref{thm.formality}]\label{thm.formalityintro}
Let $(L,d,[-,-],(-,-))$ be a quasi-cyclic DG-Lie algebra of degree $n\le 2$. Assume that  there exists a splitting $L=H\oplus d(K)\oplus K$ such that:
	\begin{enumerate}
	\item $H^i=0$ for $i<0$ (and hence also $H^i=0$ for $i>n$);\smallskip
	\item $H^0\subset L^0$ is closed with respect to the bracket $[-,-]$;\smallskip
	\item $H^i,K^i\subset L^i$ are $H^0$-submodules (with respect to the adjoint action) for all $i>0$;		
	\end{enumerate}
Then the DG-Lie algebra $(L,d,[-,-])$ is formal.		
\end{theorem}

Finally, in Section~\ref{sec.formality} we discuss the applications to moduli spaces of sheaves on minimal projective surfaces of Kodaira dimension 0. We will first prove the formality conjecture for polystable sheaves on K3 and Abelian surfaces as an immediate consequence of Theorem~\ref{thm.formalityintro}, where the polystability assumption ensures the existence of the splitting with the required properties.

Then we will extend the formality result to polystable sheaves on surfaces with torsion canonical bundle. Here the idea is to use the cyclic covering trick in order to construct the DG-Lie algebra $\RHom_X(\sF,\sF)$ as a subalgebra of a suitable quasi-cyclic DG-Lie algebra satisfying the assumptions 
of Theorem~\ref{thm.formalityintro} and then use the formality transfer theorem  due to the second named author \cite[Theorem 3.4]{dglaformality}.

\section{Review of formality and minimal models of DG-Lie algebras}
\label{sec.review}

We work over a field $\K$ of characteristic zero for the algebraic part and over the field $\C$ of complex numbers for the geometric applications. Every complex of vector spaces is intended as a cochain complex.

By definition a DG-Lie algebra $L$ is formal if it is quasi-isomorphic to its cohomology DG-Lie algebra
$H^*(L)$, equipped with the trivial differential and the induced bracket. In order to avoid possible mistakes, it is useful to keep in mind that not every DG-Lie algebra is formal and that if $L$ is formal, then in general does not exist any direct quasi-isomorphism of DG-Lie algebras $H^*(L)\to L$. However, 
since the category of DG-Lie algebras admits a model structure where the fibrations (resp.: the weak equivalences) are the surjective maps (resp.: the quasi-isomorphisms) it follows that a DG-Lie algebra $L$ is formal if and only if there exists a span of surjective quasi-isomorphisms of DG-Lie algebras 
 $L\xleftarrow{}M\xrightarrow{}H^*(L)$.

Since  two DG-Lie algebras are quasi-isomorphic if and only if they are 
weak equivalent as $L_{\infty}$ algebras we also have that a  
DG-Lie algebra $L$ is formal if and only if there exist an $L_{\infty}$ algebra $H$ and 
a span of $L_{\infty}$ weak equivalences  
\begin{equation}\label{equ.span}
 L\xleftarrow{\quad}H\xrightarrow{\quad}H^*(L)\,.
\end{equation}

We assume that the reader is familiar with the notion and basic properties of $L_{\infty}$ algebras, see e.g. \cite{getzler04,K,LadaMarkl,LadaStas,LMDT}. For reader's convenience and for fixing the sign convention, we briefly recall here the definition of $L_{\infty}$ algebra in the version used for the explicit computations that we shall perform in Section~\ref{sec.formalityforcyclicDGLA}.

Let $V$ be a graded vector space. Given $v_1,\ldots,v_n$ 
homogeneous  vectors  of $V$ and a permutation $\sigma$ of $\{1,\ldots,n\}$, we denote by 
$\chi(\sigma;v_1,\ldots,v_n)=\pm 1$ the antisymmetric Koszul sign, defined by the relation 
\[ v_{\sigma(1)}\wedge\cdots\wedge v_{\sigma(n)}=\chi(\sigma;v_1,\ldots,v_n)\,  
 v_{1}\wedge\cdots\wedge v_{n}\,\]
in the $n$th exterior power $V^{\wedge n}$. We shall simply write $\chi(\sigma)$ instead of 
$\chi(\sigma;v_1,\ldots,v_n)$ when the vectors $v_1,\ldots,v_n$ are clear from the context.
For instance if $\sigma$ is the transposition exchanging 1 and 2 we have 
$\chi(\sigma)=-(-1)^{|v_1|\,|v_2|}$ where $|v|$ denotes the degree of the homogeneous vector $v$.  
Notice that if every $v_i$ has odd degree, then 
$\chi(\sigma)=1$ for every $\sigma$.

Because of the universal property of wedge powers, we shall constantly interpret every linear map
$V^{\wedge p}\to W$ as a graded skew-symmetric $p$-linear map $V\times\cdots\times V\to W$.

\begin{definition}
An $L_{\infty}$ algebra is the data of a graded vector space  $V$ together with  a sequence of (multi)linear maps 
$\{\cdots\}_n\colon V^{\wedge n}\to V$, $n\ge 1$, such that for every $n$:
\begin{enumerate}

\item $\{\cdots\}_n$ has degree $2-n$;

\item for every $v_1,\ldots,v_n\in V$ homogeneous
\begin{equation}\label{equ.linfinitostruttura} 
\sum_{k=1}^n(-1)^{n-k}\sum_{\sigma\in S(k,n-k)}\chi(\sigma)\,
\{\{v_{\sigma(1)},\ldots,v_{\sigma(k)}\}_k,v_{\sigma(k+1)},\ldots,v_{\sigma(n)}\}_{n-k+1}=0,
\end{equation}
where $S(k,n-k)=\{\sigma\in S_n\mid \sigma(i)<\sigma(i+1),\;\forall\, i\neq k  \}$ is the set of $(k,n-k)$-shuffles.
\end{enumerate}
\end{definition}

In the above definition we used the sign convention of 
\cite{getzler04,K,LMDT}, while  in \cite{LadaMarkl,LadaStas} the maps
$\{\cdots\}_k$ differ  by the sign $(-1)^{k(k-1)/2}$. 
Every DG-Lie algebra $(L,d,[-,-])$ is an $L_{\infty}$ algebra where $\{\cdot\}_1=d$, $\{\cdot\cdot\}_2=[-,-]$ and 
$\{\cdots\}_n=0$ for every $n>2$. If $\{\cdot\}_1=0$ the $L_{\infty}$ algebra is called minimal.

There exists a general notion of $L_{\infty}$ morphism (see e.g. \cite{LMDT}), but for simplicity of exposition we only recall here the case of 
morphisms from an $L_{\infty}$ algebra to a DG-Lie algebra: this particular case will be sufficient for our purposes.

\begin{definition}\label{def.morfismo} 
Let $(V,\{\cdot\}_1,\{\cdot\,\cdot\}_2,\{\cdots\}_3,\cdots)$ be an $L_{\infty}$ algebra and $(L,d,[-,-])$ a DG-Lie algebra. An $L_{\infty}$ morphism $g\colon V\to L$ is a sequence of maps 
$g_n\colon V^{\wedge n}\to L$, $n\ge 1$, with $g_n$ of degree $1-n$ such that, for every $n$ and every 
$v_1,\ldots,v_n\in V$ homogeneous we have
\[\begin{split}
&\frac{1}{2}\sum_{p=1}^{n-1}
\!\!\!\!\!\sum_{\quad\sigma\in S(p,n-p)}\!\!\!\!\!\!\chi(\sigma) (-1)^{(1-n+p)(|v_{\sigma(1)}|+\cdots+
|v_{\sigma(p)}|-p)}
\left[g_p(v_{\sigma(1)},\ldots, v_{\sigma(p)}),\vphantom{\sum}
g_{n-p}(v_{\sigma(p+1)},\ldots,v_{\sigma(n)})\right]\\
&\quad +dg_n(v_1,\ldots,v_n)=
\sum_{k=1}^{n}(-1)^{n-k}\sum_{\sigma\in S(k,n-k)}\chi(\sigma)g_{n-k+1}(\{v_{\sigma(1)},\ldots,v_{\sigma(k)}\}_k,\ldots,v_{\sigma(n)}).\end{split}\]
\end{definition}

An $L_{\infty}$ morphism $g$ as in Definition~\ref{def.morfismo} is called a weak equivalence or a quasi-isomorphism if  $g_1\colon (V,\{\cdot\}_1)\to (L,d)$ is a quasi-isomorphism of cochain complexes.

By homotopy classification of $L_{\infty}$ algebras \cite{K}, for every DG-Lie algebra $L$ there 
exists a minimal $L_{\infty}$ algebra $H$ and an $L_{\infty}$ weak equivalence  $\imath\colon H\to L$.
The  algebra $H$ is called the $L_{\infty}$ minimal model of $L$ and it is unique up to isomorphism, while the $L_{\infty}$ morphism $i$ is unique up to homotopy. By homological perturbation theory, every splitting of the complex $(L,d)$ induces canonically a  morphism $\imath\colon H\to L$ as above. 

Recall that a splitting of $(L,d)$ is a direct sum decomposition $L=H\oplus d(K)\oplus K$ such that $H,K$ are graded vector subspaces of $L$ and the restrictions of the differential $d$ to  $H$ and $K$ are respectively zero and injective, see \cite[Section 1.4]{Wei94}.
In particular, $d(L)=d(K)$, $Z(L)=H\oplus d(K)$ and the natural map $H\to H^*(L)$ is an isomorphism of graded vector spaces.
Denoting by $\hookrightarrow$ and $\twoheadrightarrow$ the inclusions and the projections given by the splitting $L=H\oplus d(K)\oplus K$, we define  the maps
\[\imath_1\colon H\hookrightarrow L,\qquad \pi\colon L\twoheadrightarrow H,
\qquad h\colon L\twoheadrightarrow d(K) \xrightarrow{\,-d^{-1}} K \hookrightarrow L\,,\]
that satisfy the contraction identities  
\[ d\imath_1 = 0,\quad \pi d =0,\quad \pi\imath_1 = \operatorname{id}_{H},\quad dh + hd = \imath_1\pi - \operatorname{id}_L,\quad h\imath_1 = 0,\quad \pi h = 0,\qquad h^2=0\,. \]

Then, a  minimal $L_{\infty}$ algebra $(H,0,\{\cdot\cdot\}_2,\{\cdots\}_3,\ldots)$ and an extension of $\imath_1$ to an
$L_{\infty}$ quasi-isomorphism $\imath\colon H\to L$  
are defined by the recursive equations 
\begin{equation}\label{equ.transfer1} \imath_p(\xi_1,\ldots,\xi_p) = \frac{1}{2}\sum_{k=1}^{p-1}\sum_{\sigma\in S(k,p-k)}\chi(\sigma)(-1)^{\alpha(\sigma)}  h[\imath_k(\xi_{\sigma(1)},\ldots),\imath_{p-k}(\ldots,\xi_{\sigma(p)})],\quad p\ge2,
\end{equation}
\begin{equation}\label{equ.transfer2} \{\xi_1,\ldots,\xi_p\}_p = \frac{1}{2}\sum_{k=1}^{p-1}\sum_{\sigma\in S(k,p-k)} \chi(\sigma)(-1)^{\alpha(\sigma)}  \pi[\imath_k(\xi_{\sigma(1)},\ldots),\imath_{p-k}(\ldots,\xi_{\sigma(p)})], \quad p\ge2,
\end{equation}
where 
\[ \alpha(\sigma)=(1-p+k)\left(k+\sum_{i=1}^k|\xi_{\sigma(i)}|\right).\]
Notice that  for every $\xi,\eta\in H$ we have 
\[\imath_2(\xi,\eta)=h[\imath_1(\xi),\imath_1(\eta)],\qquad 
\{\xi,\eta\}_2=\pi[\imath_1(\xi),\imath_1(\eta)],\] 
the integer
$\alpha(\sigma)$ is even for $p=2$ and $\chi(\sigma)(-1)^{\alpha(\sigma)}=1$
if  $|\xi_i|$ is odd  for every $i$. Formulas \eqref{equ.transfer1} and 
\eqref{equ.transfer2} are well known and essentially dates back to Kadeishvili's paper \cite{Kad}: the choice of signs comes from standard d\'ecalage isomorphisms applied to the explicit formulas used in \cite[Thm. 3.7]{yukawate}  and \cite{LMDT}.

In \cite{dglaformality} the second named author  proved a series of formality criteria for  DG-Lie algebras. 
As a consequence of these criteria we have the following formality transfer theorem, where 
$H_{CE}^*(A,B)$ denotes the Chevalley-Eilenberg cohomology of the graded Lie algebra $A$ with coefficient in the $A$-module $B$:

\begin{theorem}[{\cite[Theorem 3.4]{dglaformality}}]\label{thm.formalitytransfer} 
Let $f\colon M\to L$ be a morphism of 
differential graded Lie algebras. Assume that 
\begin{enumerate}

\item\label{it1.thm.formalitytransfer}  $L$ is formal;

\item\label{it2.thm.formalitytransfer}  the  induced map 
$f\colon H^2_{CE}(H^*(M),H^*(M))\to H^2_{CE}(H^*(M),H^*(L))$ is injective.
\end{enumerate}
Then also $M$ is formal. 
In particular, if $L$ is formal, $f$ is injective and $f(M)$ is a direct summand of $L$ as 
$M$-module, then also $M$ is formal.
\end{theorem}

It should be noted that for $L=0$ the above theorem reduces to the classical criterion for intrinsic formality of graded Lie algebras.

\bigskip

\section{Cyclic and quasi-cyclic DG-Lie algebras}\label{section.qcyclicDGLA}

The general notion of cyclic (DG) algebra \cite{GK} specialised to DG-Lie algebras gives the following definition, see also \cite{LS}.

\begin{definition}\label{def.cyclic} 
Let $n$ be an integer. A \emph{cyclic DG-Lie algebra} $(L,d,[-,-],(-,-))$ of degree $n$ is a finite dimensional DG-Lie algebra $(L,d,[-,-])$ equipped with a degree $-n$ non-degenerate graded symmetric bilinear form $(-,-)\colon L^{\odot2}\to\K[-n]$ such 
that  
\[ (dx,y)=(-1)^{|x|+1}(x,dy),\qquad ([x,y],z) = (x,[y,z]),\qquad\forall\,x,y,z\in L\,. \]
\end{definition}
	
The condition $(dx,y)=\pm(x,dy)$ implies in particular that 
$d(L)^{\perp}=\ker(d)$; since $L$ is finite dimensional we have 
$d(L)=\ker(d)^{\perp}$ and this implies that also the induced bilinear form in the cohomology 
$H^*(L)$ is non-degenerate.

\begin{example}[Symplectic representations]
Let $(V,\omega)$ be a finite dimensional symplectic vector space and let $\mathfrak{g}$ be a finite dimensional Lie algebra. Recall that a left action 
\[ \mathfrak{g}\times V\to V,\qquad (g,v)\mapsto gv,\]
is called symplectic if for every $v,w\in V$ and every $g\in\mathfrak{g}$  we have
\[\omega(gv,w)+\omega(v,gw)=0\,.\]
There exists a natural correspondence between (isomorphism classes of) symplectic representations and (isomorphism classes of) cyclic DG-Lie algebras of degree 2 with trivial differential and without elements of negative degree:
given a symplectic action as above consider the graded Lie algebra $L=H(L)=L^0\oplus L^1\oplus L^2$ whose cyclic Lie structure is defined as follows.

\begin{enumerate}
\item $L^0=\mathfrak{g}$, $L^1=V$, and $L^2=\mathfrak{g}^{\vee}=\Hom_{\K}(\mathfrak{g},\K)$.

\item The Lie bracket is defined by
\begin{itemize}
\item $[g,v]=gv$ for every $g\in L^0$, $v\in L^1$,
\item $[v,w] \colon h\mapsto \omega(hv,w)$ for every $v,w\in L^1$, $h\in \mathfrak{g}$,
\item $[g,y]\colon h\mapsto y([h,g]_{\mathfrak{g}})$ for every $g,h\in\mathfrak{g}$, $y\in\mathfrak{g}^{\vee}$.
\end{itemize}

\item The pairing is defined by
\begin{itemize} 
\item $(-,-)\colon L^0\times L^2\to \K$ is the natural pairing,
\item $(v,w)=\omega(v,w)$ for every $v,w\in L^1$.
\end{itemize}
\end{enumerate}
The relations below easily follow from the above conditions:
\[ (h,[g,y]_L)=([h,g]_{\mathfrak{g}},y) \qquad \mbox{ for every } h,g\in L^0,\, y\in L^2\, , \]
\[ (g,[v,w]_L)=\omega(gv,w)=\omega([g,v]_L,w) \qquad \mbox{ for every } g\in L^0,\, v,w\in L^1\, . \]
Moreover, the equalities
\[ \omega(gv,w) + \omega(v,gw) = (gv, w) -\omega(gw,v) = ([g,v]_L,w) - (g,[v,w]_L) \qquad \mbox{ for every } g\in\mathfrak{g}, \, v,w\in L^1 \]
show that the \emph{symplectic condition} $\omega(gv,w)+\omega(v,gw)=0$ is equivalent to the \emph{cyclicity condition} $(g,[v,w])=([g,v],w)$.
The proof that the above defined data $(L,0,[-,-],(-,-))$ provides an example of a cyclic DG-Lie algebra of degree 2 is now straightforward. 

Notice that the Maurer-Cartan functional  $\frac{1}{2}[v,v]$ coincides by definition with the moment map $\mu\colon V\to \mathfrak{g}^{\vee}$ of the symplectic representation. 
\end{example}

\begin{example}\label{ex.nocontraction} 
Consider the following complex of vector spaces in degrees 0,1,2:
\[ L\colon\qquad \Span(a,b)\xrightarrow{d} \Span(x,y,p,db)\xrightarrow{d}\Span(z,dp)\]
equipped with 
the bilinear form $(-,-)\colon L^{\odot 2}\to \K[-2]$,  where the only nontrivial products between basis vectors are:
\[ (x,y)=-(y,x)=-1,\quad (db,p)=-(p,db)=-1,\quad (a,z)=(z,a)=1,\quad (b,dp)=(dp,b)=1\,.\]
Next consider the bracket $[-,-]\colon L^{\wedge 2}\to L$,  
 where the only nontrivial brackets between basis vectors are:
\[ [a,x]=db,\quad [a,p]=y,\quad [x,x]=dp,\quad [p,x]=z,\quad [b,x]=y\,.\]
The next proposition summarises the properties or the above example that are relevant for this paper.  

\begin{proposition}\label{prop.esempiononformale} 
In the above setup:
\begin{enumerate}

\item $L$ is a cyclic DG-Lie algebra of degree 2;

\item $L$ is not a formal DG-Lie algebra;

\item  there does not exist any splitting  $L=H\oplus d(K)\oplus K$ 
such that $[H^0,H^1]\subset H^1$.

\end{enumerate}
\end{proposition}

\begin{proof} The first item is a tedious but straightforward computation. For the second item we observe that the triple Massey power of $x$ is nontrivial since
$dp=[x,x]$ and $[p,x]=z$. The third item is clear since  for every splitting there exists $\alpha\in \K$ such that 
$x+\alpha db\in H^1$  and therefore
$[a,x+\alpha db]=[a,x]=db\not\in H^1$. 
\end{proof}
\end{example}

\begin{example}\label{ex.noformaldimension3} 
Consider the following complex of vector spaces in degrees 1,2:
\[ L:\qquad \Span(a,b)\xrightarrow{d} \Span(x,db)\]
equipped with 
the closed bilinear form $(-,-)\colon L^{\odot 2}\to \K[-3]$,  where the only nontrivial products between basis vectors are $(a,x)=(b,db)=1$.
Next consider the bracket $[-,-]\colon L^{\wedge 2}\to L$,  
where the only nontrivial brackets between basis vectors are $[a,a]=db$, $[a,b]=x$.
The same argument used in the proof of Proposition~\ref{prop.esempiononformale} shows that $L$ is a cyclic non-formal DG-Lie algebra of degree 3. 
\end{example}

It is useful to enlarge the class of cyclic DG-Lie algebras by removing the assumption that
$L$ is finite dimensional, which is not satisfied in most geometrical situations. 
The same weakening of assumption was considered by Kontsevich \cite{K94} in the associative case.

\begin{definition}\label{def.qcyclicdgla}
A \emph{quasi-cyclic DG Lie algebra} $(L,d,[-,-],(-,-))$ of degree $n$ is a DG-Lie algebra $(L,d,[-,-])$ with finite dimensional cohomology, together with a degree $-n$ symmetric bilinear form $(-,-)\colon L^{\odot2}\to\K[-n]$ which satisfies 
\[ (dx,y)=(-1)^{|x|+1}(x,dy),\qquad ([x,y],z) = (x,[y,z]),\qquad\forall\,x,y,z\in L\,. \]
and such that the induced form $(-,-)\colon H(L)^{\odot2} \to\K[-n]$ is non-degenerate.
\end{definition}

If $(L,d,[-,-],(-,-))$ is a quasi-cyclic DG-Lie algebra then its cohomology $H^*(L)$ is naturally endowed with a structure of cyclic graded Lie algebra of the same degree.

\begin{example}[Vector bundles on manifolds with trivial canonical bundle]\label{ex.VBonK3}  
Let $\sE$ be a locally free  sheaf of a smooth complex projective manifold $X$ of dimension $n$ with trivial canonical bundle, and denote by $\omega_X$ be a 
holomorphic volume form. Then 
the Dolbeault complex 
\[ L=A^{0,*}_X(\HOM(\sE,\sE))\]
of the sheaf of endomorphisms of $\sE$ is a quasi-cyclic DG-Lie algebra of degree $n$, where
\[ (f,g)=\int_X \omega_X\wedge \Tr(fg)\,.\]
We have $H^i(L)=\Ext^i_X(\sE,\sE)$ and by Serre duality the induced pairing 
\[(-,-)\colon \Ext^i_X(\sE,\sE)\times \Ext^{n-i}_X(\sE,\sE)\to \C\] 
is non-degenerate. In Section~\ref{sec.formality} we extend this construction to coherent sheaves. 
\end{example}

We are now ready to state one of the main results of this paper, namely a 
sufficient condition for formality of quasi-cyclic DG-Lie algebras of degree $\le 2$. 
		
\begin{theorem}\label{thm.formality} 
	Let $(L,d,[-,-],(-,-))$ be a quasi-cyclic DG-Lie algebra of degree $n\le 2$. Assume that  there exists a splitting $L=H\oplus d(K)\oplus K$ such that:
		\begin{enumerate} 
		
		\item $H^i=0$ for $i<0$ (and hence also $H^i=0$ for $i>n$);\smallskip
		
		\item\label{item:hyp1} $H^0\subset L^0$ is closed with respect to the bracket $[-,-]$;\smallskip
			
		\item\label{item:hyp2} $H^i,K^i\subset L^i$ are $H^0$-submodules (with respect to the adjoint action) for all $i>0$;
			
		\end{enumerate}
		Then the DG-Lie algebra $(L,d,[-,-])$ is formal.
		
	\end{theorem}
		
For $n\le 0$ the above theorem is trivial since $H^i=0$ for every $i>0$ and then 
the embedding $H^0\to L$ is a quasi-isomorphism of DG-Lie algebras.   		
The next Section~\ref{sec.formalityforcyclicDGLA} will be entirely devoted to the  
(long) proof in the case $n=2$, whose first step also provides a complete proof for $n=1$. 
Examples~\ref{ex.noformaldimension3} and 
\ref{ex.nocontraction} show that formality fails if either $n>2$ or   
without the assumption \eqref{item:hyp2}, even for cyclic DG-Lie algebras.

\bigskip	
\section{Proof of Theorem~\ref{thm.formality}}	
\label{sec.formalityforcyclicDGLA}

Let $(L,d,[-,-],(-,-))$ be as in  Theorem~\ref{thm.formality}. If $A,B$ are two subsets of $L$ we shall write $A\perp B$ if $(x,y)=0$ for every $x\in A$, $y\in B$. For instance, it follows immediately from the relation $(dx,y)=(-1)^{|x|+1}(x,dy)$ that $d(K)\perp d(K)$ and $H\perp d(K)$.

\begin{lemma}\label{lem.semiorto} 
Up to a possible restriction to a quasi-isomorphic DG-Lie subalgebra of $L$ we may 
assume that the splitting $L=H\oplus d(K)\oplus K$ satisfies the following conditions:
		\begin{enumerate} 
		
		\item $H^i=0$ for $i<0$;
		
		\item\label{item:hyp3} $H^0\subset L^0$ is closed with respect to the bracket $[-,-]$;
		
		\item\label{item:hyp4} $H^i,K^i\subset L^i$ are $H^0$-submodules (with respect to the adjoint action) for all $i\in \Z$;
		
		\item\label{item:hyp5} $H\perp K$.
		\end{enumerate}
		\end{lemma}
	
\begin{proof} Since  $H^i=0$ for every $i<0$  
the  DG-Lie subalgebra 
\[ H^0\oplus (H^1\oplus K^1)\oplus  (H^2\oplus d(K^1)\oplus K^2)\oplus \cdots\;\]
is quasi-cyclic and quasi-isomorphic to $L$. This proves that, up to a possible 
restriction to a quasi-isomorphic DG-Lie subalgebra it is not restrictive to assume the validity of 
condition \eqref{item:hyp4}.
Next, for every integer $i$ consider the vector subspace 
\[ C^i=\{x\in H^i\oplus K^i\mid (x,y)=0\quad\forall\; y\in H^{n-i}\}.\]
Since $(-,-)\colon H^i\times H^{n-i}\to \K$ is a perfect pairing, the map
\[ \begin{aligned}
H^i\oplus C^i &\to H^i\oplus K^i \\
(h_1,h_2,k) &\mapsto (h_1-h_2,k)
\end{aligned} \]
is an isomorphism.
%we have $H^i\oplus K^i=H^i\oplus C^i$.
If $x\in C^i$ and $a\in H^0$, then $[a,x]\in H^i\oplus K^i$;  for every
$y\in H^{n-i}$ we have $(y,[a,x])=([y,a],x)=0$ and therefore $[a,x]\in C^i$. Finally, replacing 
$K^i$ with $C^i$ we may assume $H\perp K$. 

For later use it should  be pointed out that  the 
non-degeneneracy of $(-,-)\colon H^{\odot 2}\to \K$ immediately implies 
\begin{equation}\label{equ.hyp6}
x\in K\oplus d(K)\iff (x,y)=0\quad\text{for every}\quad y\in H\,.
\end{equation}
\end{proof}

From now on we assume that $(L,d,[-,-],(-,-))$ is a quasi-cyclic DG-Lie algebra of degree $n\le 2$ equipped with a splitting $L=H\oplus d(K)\oplus K$ satisfying 
the conditions of Lemma~\ref{lem.semiorto}.
The fist step is to use such a splitting  in order to produce a minimal $L_{\infty}$ model of 
$L$. Following the recipe described in Section~\ref{sec.review} we introduce the maps 
\[\imath_1\colon H\hookrightarrow L,\qquad \pi\colon L\twoheadrightarrow H,
\qquad h\colon L\twoheadrightarrow d(K) \xrightarrow{-d^{-1}} K \hookrightarrow L\,\] 
that satisfy the relations 
\[ (\imath_1(x),\imath_1(y)) = (x,y),\qquad (h(l),\imath_1(x)) = 0,\quad(\pi(l),x) = (l,\imath_1(x)),\qquad\forall\,x,y\in H,\;l\in L\,. \]
The first one is obvious and the second one follows from the orthogonality condition 
$H\perp K$. 
Since $\operatorname{Im}(\imath_1)=H$,   $\operatorname{Im}(h)=K$ and $H\perp d(K)\oplus K$,  the first two imply the third: 
\[(\pi(l),x)= (\imath_1\pi(l),\imath_1(x))=((\operatorname{id}_L+dh+hd)(l),\imath_1(x))=(l,\imath_1(x))\,.\]

The maps  $\imath_1,\pi,h$ induce via homotopy transfer a minimal  $L_\infty$-algebra structure on $H$, together with an $L_\infty$ quasi-isomorphism $\imath\colon H(L)\to L$ of $L_\infty$-algebras with linear part $\imath_1$. 
The quadratic components are given by 
\[ \imath_2(\xi_1,\xi_2)=h[\imath(\xi_1),\imath(\xi_2)],\qquad \{\xi_1,\xi_2\}_2=
\pi[\imath(\xi_1),\imath(\xi_2)]\,,\]
while the higher brackets $\{\cdots\}_p\colon H^{\wedge p}\to H[2-p]$ and the higher Taylor coefficients $\imath_p\colon H^{\wedge p}\to L[1-p]$, $p\ge2$, are explicitly (and recursively) defined by
\begin{equation}\label{eq:transfer1} \imath_p(\xi_1,\ldots,\xi_p) = \frac{1}{2}\sum_{k=1}^{p-1}\sum_{\sigma\in S(k,p-k)} \pm\, h[\imath_k(\xi_\sigma(1),\ldots),\imath_{p-k}(\ldots,\xi_{\sigma(p)})],
\end{equation}
\begin{equation}\label{eq:transfer2} \{\xi_1,\ldots,\xi_p\}_p = \frac{1}{2}\sum_{k=1}^{p-1}\sum_{\sigma\in S(k,p-k)} \pm\, \pi[\imath_k(\xi_\sigma(1),\ldots),\imath_{p-k}(\ldots,\xi_{\sigma(p)})],
\end{equation}
where $\pm$ is the appropriate Koszul sign described explicitly in \eqref{equ.transfer1} and \eqref{equ.transfer2}. These signs will simplify in our specific case, for instance $\pm1=+1$ whenever $\xi_i\in H^1$ for every $i$, and we don't need to make them explicit.

Notice that $\{a,b\}_2=[a,b]$ for $a,b\in H^0$ and under the natural isomorphism $H\cong H^*(L)$, the quadratic bracket $\{x,y\}_2=\pi[\imath_1(x),\imath_1(y)]$ on $H$ is just the bracket induced by $[-,-]$ in cohomology.

\begin{lemma}\label{lem.secondvanishing}
In the above setup, for every $p\ge 2$ and every $g\in H^0$ we have 
\begin{equation*} \imath_p(g,\ldots)=0,\qquad \{g,\ldots\}_{p+1}=0 \,.
\end{equation*}
\end{lemma}
\begin{proof}
If $g\in H^0$, then $[\imath_1(g),\imath_1(\xi)]\in H \subset \operatorname{Ker}(h)$ for all $\xi\in H$, since $H$ is a $H^0$-submodule of $L$, thus $\imath_2(g,\xi)=0$ for all $g\in H^0$ and $\xi\in H$. In general, by formulas \eqref{eq:transfer1}, \eqref{eq:transfer2} and induction on $p$, for all $p\ge2$, $g\in H^0$ and $\xi_1,\ldots,\xi_p\in H$, we have 
\[\{\xi_1,\ldots,\xi_p,g\}_{p+1}= \pm\pi[\imath_p(\xi_1,\ldots,\xi_p),\imath_1(g)],\qquad \imath_{p+1}(\xi_1,\ldots,\xi_p,g)=\pm h[\imath_p(\xi_1,\ldots,\xi_p),\imath_1(g)]\,. \]
Finally, notice that for $p\ge2$ we have $\operatorname{Im}(\imath_p)\subset K\subset \operatorname{Ker}(h)\bigcap \operatorname{Ker}(\pi)$, and that $K$ is by hypothesis a $H^0$-submodule of $L$. This implies that  $[\imath_p(\xi_1,\ldots,\xi_p),\imath_1(g)]\in K$ and therefore
\[\{\xi_1,\ldots,\xi_p,g\}_{p+1}=\imath_{p+1}(\xi_1,\ldots,\xi_p,g)=0\,.\] 
\end{proof}

Lemma~\ref{lem.secondvanishing} provides a complete proof of formality  for $n=1$, since 
$H^i=0$ for every $i\not=0,1$ and therefore by degree reasons $\{\cdots\}_{p+1}=0$ for every 
$p\ge 2$. From now on we assume that \emph{the degree of the quasi-cyclic DG-Lie algebra $L$ of Theorem~\ref{thm.formality} is  equal to $n=2$.}

\begin{lemma}\label{lem.firstvanishing} 
In the above situation,  for every $\xi_1,\ldots,\xi_p\in H^1$, $p\ge 3$, we have 
\begin{equation}\label{equ.triplevanish}
\pi[\imath_p(\xi_1,\ldots,\xi_{p-1}),\imath_1(\xi_{p})]=0,
\end{equation}
\begin{equation}\label{eq:transfer3} \{\xi_1,\ldots,\xi_p\}_p = \frac{1}{2}\sum_{k=2}^{p-2}\;\sum_{\sigma\in S(k,p-k)} \pi[\imath_k(\xi_\sigma(1),\ldots),\imath_{p-k}(\ldots,\xi_{\sigma(p)})]\,,
\end{equation}
and therefore $\{\xi_1,\xi_2,\xi_3\}_3=0$ for every $\xi_1,\xi_2,\xi_3\in H^1$. 
\end{lemma}

\begin{proof} It is sufficient to prove \eqref{equ.triplevanish}.
We first note that  the image of $\imath_1$ is contained in $H$ and the image of $\imath_j$ is contained in $K$ for every $j>1$. Moreover we can rewrite \eqref{equ.hyp6} in the form 
$\pi(x)=0$ iff $(x,y)=0$ for every $y\in H$: now it is sufficient to observe that for any 
$g\in H^0$, $x\in H^1$ and $y\in K^1$ we have 
$(g,[x,y])=([g,x],y)=0$ since $H^1$ is an $H^0$-module and $K^1$ is orthogonal to $H^1$.
\end{proof}

By degree reasons, Lemma~\ref{lem.secondvanishing} implies that for $p\ge2$ we have $\{\xi_1,\cdots,\xi_{p+1}\}_{p+1}=0$ unless $\xi_1,\ldots,\xi_{p+1}\in H^1$, and then by Lemma~\ref{lem.firstvanishing} we have $\{\cdots\}_3\equiv0$.
However, it should be noted that  in general the higher brackets $\{\cdots\}_p$ will not vanish 
for $p\ge4$ and therefore the proof of Theorem~\ref{thm.formality} is still very far to be concluded.

\textbf{Notation:} from now on we shall denote by $\mathfrak{g}$ the Lie algebra $(H^0,\{-,-\})$.\\

Now we notice that the item \eqref{item:hyp4} in the hypotheses of Lemma~\ref{lem.semiorto} implies that the maps $\imath\colon H\to L$, $\pi\colon L\to H$ and $h:L\to L[-1]$ are equivariant with respect to the induced $\mathfrak{g}$-module structures.

Since $\imath=(\imath_1,\imath_2,\ldots)$ is a morphism of $L_\infty$ algebras we have
\begin{multline*}\sum_{k=1}^{p}\sum_{\sigma\in S(p+2-k,k-1)} 
\pm \imath_k(\{\xi_{\sigma(1)},\cdots\}_{p+2-k},\ldots,\xi_{\sigma(p+1)}) = \\ 
= \pm d\imath_{p+1}(\xi_1,\ldots,\xi_{p+1})+\frac{1}{2}\sum_{j=1}^{p}\sum_{\sigma\in S(j,p+1-j)}\pm[  \imath_j(\xi_{\sigma(1)},\ldots),\imath_{p+1-j}(\ldots,\xi_{\sigma(p+1)}) ],
\end{multline*}
and taking $p\ge 2$, $\xi_1,\ldots,\xi_p\in H^1$, $\xi_{p+1}=g\in\mathfrak{g}$, by Lemma~\ref{lem.secondvanishing} the above expression reduces to
\begin{equation}\label{eq:equivariance_for_i}  [\imath_p(\xi_1,\ldots,\xi_p),\imath_1(g)] = \imath_p(\{\xi_1,g \}_2,\ldots,\xi_p)+\cdots+\imath_{p}(\xi_1,\ldots,\{\xi_p,g\}_2)\,. 
\end{equation}

Notice that the formula \eqref{eq:equivariance_for_i} is trivially satisfied  also for $p=1$. For later use it is useful to introduce, for every $0<j<p$ the  function
\[ I_j^p\colon (H^1)^{\odot j}\otimes (H^1)^{\odot p-j}\to \K,\quad 
I_j^p(\xi_1,\ldots,\xi_p)=(\imath_j(\xi_1,\ldots,\xi_j),\imath_{p-j}(\xi_{j+1},\ldots,\xi_p)).\]
Then for every $0<j<p$, $\xi_1,\ldots,\xi_p\in H^1$ and $g\in \mathfrak{g}$ we have 
\begin{equation}\label{equ.sommaI_j}
\sum_{i=1}^pI_j^p(\xi_1,\ldots,\{\xi_i,g\}_2,\ldots,\xi_p)=0\,.
\end{equation}
The proof of \eqref{equ.sommaI_j} is an immediate consequence of \eqref{eq:equivariance_for_i} together with the identity $([l_1,\imath_1(g)],l_2)+(l_1,[l_2,\imath_1(g)])=0$  for all $g\in\mathfrak{g}$, $l_1,l_2\in L^1$.  Moreover, the orthogonality condition $H\perp K$ implies that for every $p\ge 2$ we have $I^{p+1}_1=I^{p+1}_p=0$.

\medskip
\textbf{Notation:} 
we denote by $\{-,-\}\colon H^*(L)^{\wedge 2}\to H^*(L)$ the Lie bracket induced by the one 
$[-,-]\colon L^{\wedge2}\to L$ on $L$. We have already observed that via the natural identification $H=H^*(L)$ we have 
$\{-,-\}=\{\cdot\cdot\}_2$ and   
it is straightforward to check that it continues to satisfy $(\{x,y\},z)=(x,\{y,z\})$ for all $x,y,z\in H^*(L)$.\\

By homotopy classification of DG-Lie and $L_\infty$ algebras, in order to prove the formality of $L$  it is enough to exhibit an $L_\infty$ isomorphism
\[ f\colon(H,0,\{-,-\},0,\{\cdots\}_4,\{\cdots\}_5,\ldots)\to (H^*(L),0,\{-,-\},0,0,\ldots)  \]
between $H$ with the transferred $L_\infty$ algebra structure and $H^*(L)$ with the induced graded Lie algebra structure. Denoting by $f_p\colon H^{\wedge p}\to H[1-p]$ the Taylor coefficients of $f$, the necessary relations these have to satisfy in order for $f$ to be an $L_\infty$ morphism read
\begin{multline}\label{eq:Loo} \sum_{k=1}^{p}\sum_{\sigma\in S(p+2-k,k-1)} \pm f_k(\{\xi_{\sigma(1)},\cdots\},\ldots,\xi_{\sigma(p+1)}) = \\ = \frac{1}{2}\sum_{j=1}^{p}\sum_{\sigma\in S(j,p+1-j)}\pm\{ f_j(\xi_{\sigma(1)},\ldots),f_{p+1-j}(\ldots,\xi_{\sigma(p+1)}) \}
\end{multline}
for all $p\ge2$ and $\xi_1,\ldots,\xi_{p+1}\in H$. If these are satisfied, for $f$ to be an isomorphism of $L_\infty$ algebras it is necessary and sufficient that its linear part 
$f_1\colon H\to H$ is an isomorphism of graded spaces. We look for an $L_\infty$ isomorphism 
$f$ as above such that moreover $f_1=\operatorname{id}_H$ and $f_p(\xi_1,\ldots,\xi_p)=0$ for $p\ge2$ unless $p\ge 3$ and $\xi_1,\ldots,\xi_p\in H^1$. With these hypotheses, many of the previous relations \eqref{eq:Loo} become trivial, and the only non-trivial ones we are left to verify are 
\begin{equation}\label{eq:Loo1} \left\{ f_p(\xi_1,\ldots,\xi_p),g \right\} = 
f_p\left(\left\{ \xi_1,g\right\},\ldots,\xi_p\right)+\cdots +f_p\left(\xi_1,\ldots,\left\{\xi_p,g\right\}\right), 
\end{equation}
\begin{equation}\label{eq:Loo2} \left\{ \xi_1,\ldots,\xi_{p+1} \right\}_{p+1} = \frac{1}{2}\sum_{j=1}^{p}\sum_{\sigma\in S(j,p+1-j)}\left\{ f_j(\xi_{\sigma(1)},\ldots),f_{p+1-j}(\ldots,\xi_{\sigma(p+1)})\right\},
\end{equation}
for all $p\ge2$, $\xi_1,\ldots,\xi_{p+1}\in H^1$ and $g\in\mathfrak{g}$ (as in the case of transfer formulas, Koszul signs have disappeared since $|\xi_1|=\cdots=|\xi_{p+1}|=1$). 
Since $f_2=0$ by definition and we  already know that 
$\{\cdots\}_3=0$ the relations \eqref{eq:Loo1} and \eqref{eq:Loo2} are trivially satisfied 
for $p=2$.

For every $p\ge3$ and every $1<j<p$ we define recursively  the linear maps
\[f_p\colon(H^1)^{\odot p}\to H^1,\qquad 
F^{p+1}_j\colon (H^1)^{\odot j}\otimes (H^1)^{\odot p-j+1}\to \K\,,\]
by the formulas 
\begin{equation}
F^{p+1}_j(\xi_1,\ldots,\xi_{p+1})=(f_j(\xi_{1},\ldots,\xi_j),f_{p-j+1}(\xi_{j+1},\ldots,\xi_{p+1}))\,,
\end{equation} 
\begin{equation}\label{eq:Looformula}	\left(f_p(\xi_1,\ldots,\xi_p),\xi_{p+1}\right)=
\frac{1}{2} \sum_{j=2}^{p-1}\sum_{\sigma\in S(j,p-j)}   
(I^{p+1}_j-F^{p+1}_j)(\xi_{\sigma(1)},\ldots,\xi_{\sigma(p)},\xi_{p+1}).
\end{equation}
for all $\xi_1,\ldots,\xi_{p+1}\in H^1$. 
The validity of \eqref{eq:Loo1} is proved in the following lemma.
	
\begin{lemma} In the above situation, for every $p\ge 2$, every $1<j<p$, every 	
$\xi_1,\ldots,\xi_{p+1}\in H^1$ and every  $g\in\mathfrak{g}$ we have

\begin{equation}\label{equ.boo}
\sum_{i=1}^{p}f_{p}(\xi_1,\ldots,\{\xi_i,g\},\ldots,\xi_{p})=\{ f_p(\xi_1,\ldots,\xi_p),g \},
\end{equation}
\begin{equation}\label{equ.sommaF_j}
\sum_{i=1}^{p+1}F_j^{p+1}(\xi_1,\ldots,\{\xi_i,g\},\ldots,\xi_{p+1})=0\,.
\end{equation}
\end{lemma}

\begin{proof} The above formula are trivially satisfied  for $p=2$, since $f_2=0$ and  \eqref{equ.sommaF_j} is empty.  Assuming  \eqref{equ.boo} valid for all integers smaller than $p$ we have 
\[ \begin{split}
&\sum_{i=1}^{p+1}F_j^{p+1}(\xi_1,\ldots,\{\xi_i,g\},\ldots,\xi_{p+1})\\
&\quad=(\{f_j(\xi_1,\ldots,\xi_j),g\},f_{p-j-1}(\xi_1,\ldots,\xi_j))+
(f_j(\xi_1,\ldots,\xi_j),\{f_{p-j-1}(\xi_1,\ldots,\xi_j),g\})=0\,,\end{split}\]
where the second equality follows from the cyclic condition 
$(\{x,g\},y)+(x,\{y,g\})=0$ for all $g\in\mathfrak{g}$ and $x,y\in H^1$.
For the same reason we have

\[\begin{split} 	
\left(\left\{f_p(\xi_1,\ldots,\xi_p),g\right\},\xi_{p+1}\right)&= - \left(f_p(\xi_1,\ldots,\xi_p),\left\{\xi_{p+1},g\right\}\right)\\ 
&= \left(- \frac{1}{2}\right) \sum_{j=2}^{p-1}\sum_{\sigma\in S(j,p-j)}     
(I^{p+1}_j-F^{p+1}_j)(\xi_{\sigma(1)},\ldots,\xi_{\sigma(p)},\left\{\xi_{p+1},g\right\})\\
&=\frac{1}{2}\sum_{i=1}^p\sum_{j=2}^{p-1}\sum_{\sigma\in S(j,p-j)}     
(I^{p+1}_j-F^{p+1}_j)(\xi_{\sigma(1)},\ldots,\left\{\xi_{\sigma(i)},g\right\},\ldots,\xi_{p+1})\\
&=\sum_{h=1}^p  \left(f_p(\xi_1,\ldots,\{\xi_h,g\},\ldots,\xi_p),\xi_{p+1}\right),
\end{split}\]         
where in the second and the fourth  equalities we have used 
by the defining formula \eqref{eq:Looformula}, while the third equality 
is a consquence of \eqref{equ.sommaI_j} and \eqref{equ.sommaF_j}.  Since $(-,-)$ is non-degenerate in $H^1$, the formula 
\[ \left(\left\{f_p(\xi_1,\ldots,\xi_p),g\right\},\xi_{p+1}\right)=
\sum_{h=1}^p  \left(f_p(\xi_1,\ldots,\{\xi_h,g\},\ldots,\xi_p),\xi_{p+1}\right)\,\]
is completely equivalent to \eqref{equ.boo}.
\end{proof}

Finally we prove equation \eqref{eq:Loo2},
or equivalently that 
 \begin{equation*}\label{eq:Loo2bis} 2\left(\left\{ \xi_1,\ldots,\xi_{p+1} \right\}_{p+1},g\right) = \sum_{j=1}^{p}\sum_{\sigma\in S(j,p+1-j)}(\left\{ f_j(\xi_{\sigma(1)},\ldots),f_{p+1-j}(\ldots,\xi_{\sigma(p+1)})\right\},g),
\end{equation*}
for every $\xi_1,\ldots,\xi_{p+1}\in H^1$ and $g\in\mathfrak{g}$. By Equation~\eqref{eq:transfer3} we have 
\[\begin{split} 
2\left(\left\{ \xi_1,\ldots,\xi_{p+1} \right\}_{p+1}, g \right) &= \sum_{j=2}^{p-1}\sum_{\sigma\in S(j,p+1-j)}\left(\pi\left[ \imath_j\left(\xi_{\sigma(1)},\ldots\right) ,\imath_{p+1-j}\left(\ldots,\xi_{\sigma(p+1)}\right) \right],g\right)\\
&= \sum_{j=2}^{p-1}\sum_{\sigma\in S(j,p+1-j)}\left(\left[ \imath_j\left(\xi_{\sigma(1)},\ldots\right) ,\imath_{p+1-j}\left(\ldots,\xi_{\sigma(p+1)}\right) \right],\imath_1(g)\right) \; .
\end{split}
\]
By using the cyclic relation $([l_1,l_2],l_3)=(l_1,[l_2,l_3])$, $\forall\;l_1,l_2,l_3\in L$, and Equation \eqref{eq:equivariance_for_i} we get

\[ \begin{split} 2\left(\left\{ \xi_1,\ldots,\xi_{p+1} \right\}_{p+1}, g \right) &= 
\sum_{j=2}^{p-1}\sum_{\sigma\in S(j,p+1-j)}\left(\imath_j\left(\xi_{\sigma(1)},\ldots\right) ,[\imath_{p+1-j}\left(\ldots,\xi_{\sigma(p+1)}\right),\imath(g)]\right)  \\ 
 &= \sum_{j=2}^{p-1}\sum_{\sigma\in S(j,p-j,1)}\left( \imath_j\left(\xi_{\sigma(1)},\ldots\right),\imath_{p+1-j}\left(\ldots,\xi_{\sigma(p)}, \left\{ \xi_{\sigma(p+1)},g \right\}\right) \right)\\
&=\sum_{j=2}^{p-1}\sum_{\sigma\in S(j,p-j,1)}I^{p+1}_j\left(\xi_{\sigma(1)},\ldots,
\xi_{\sigma(p)}, \left\{ \xi_{\sigma(p+1)},g \right\}\right).  
	\end{split}\]
where 	$S(j,p-j,1)$ is the set of permutations $\sigma$ of $1,\ldots,p+1$ such that 
\[ \sigma(1)<\cdots<\sigma(j),\qquad \sigma(j+1)<\cdots<\sigma(p)\,.\]
On the other side, 
\[\begin{split}
\sum_{j=1}^{p}&\sum_{\sigma\in S(j,p+1-j)}\left( \left\{ f_j(\xi_{\sigma(1)},\ldots),f_{p+1-j}(\ldots,\xi_{\sigma(p+1)})\right\},g\right)\\  
&=\sum_{j=1}^{p}\sum_{\sigma\in S(j,p+1-j)}
\left( f_j(\xi_{\sigma(1)},\ldots),\{f_{p+1-j}(\ldots,\xi_{\sigma(p+1)}),g\}\right)\\
&= \sum_{j=2}^{p-1}\sum_{\sigma\in S(j,p-j,1)} \left( f_j(\xi_{\sigma(1)},\ldots),\{f_{p+1-j}(\ldots,\xi_{\sigma(p+1)}),g\}\right)\\ 
&\qquad\quad+\sum_{\sigma\in S(p,1)} \left( f_p\left(\xi_{\sigma(1)},\ldots,\xi_{\sigma(p)}\right), \left\{\xi_{\sigma(p+1)},g\right\} \right). 
\end{split}\]
Reasoning as before and using the already proved Equation \eqref{eq:Loo1} we have
\[\begin{split}
\sum_{j=1}^{p}\sum_{\sigma\in S(j,p+1-j)}&\left( \left\{ f_j(\xi_{\sigma(1)},\ldots),f_{p+1-j}(\ldots,\xi_{\sigma(p+1)})\right\},g\right)\\  
&= \sum_{j=2}^{p-1}\sum_{\sigma\in S(j,p-j,1)} \left( f_j\left(\xi_{\sigma(1)},\ldots  \right), f_{p+1-j}\left(\ldots, \xi_{\sigma(p)},\left\{ \xi_{\sigma(p+1)},g\right\}\right)  \right)\\ 
&\qquad+\sum_{\sigma\in S(p,1)} \left( f_p\left(\xi_{\sigma(1)},\ldots,\xi_{\sigma(p)}\right), \left\{\xi_{\sigma(p+1)},g\right\} \right) \\ 
&= \sum_{j=2}^{p-1}\sum_{\sigma\in S(j,p-j,1)} F^{p+1}_j\left(\xi_{\sigma(1)},\ldots,\xi_{\sigma(p)},
\left\{ \xi_{\sigma(p+1)},g\right\}\right) \\ 
&\qquad+\sum_{\sigma\in S(p,1)} \left( f_p\left(\xi_{\sigma(1)},\ldots,\xi_{\sigma(p)}\right), \left\{\xi_{\sigma(p+1)},g\right\} \right) \\ 
&=\sum_{j=2}^{p-1}\sum_{\sigma\in S(j,p-j,1)} I^{p+1}_j\left(\xi_{\sigma(1)},\ldots,
\xi_{\sigma(p)}, \left\{ \xi_{\sigma(p+1)},g \right\}\right),
\end{split}
\]
where in the last equality we used the recursive definition \eqref{eq:Looformula} of $f_p$.
The proof of Theorem~\ref{thm.formality} is now complete.

%\begin{remark} 
%The above formality theorem admits a simpler proof under the additional assumption that 
%the abstract Hodge decomposition satisfy the orthogonality conditions $H\perp K$ and $K\perp K$. Unfortunately, contrary to what implicitly (and in some case also explicitly) asserted in several published papers, the condition  
%$K\perp K$ cannot be satisfied: consider for instance an acyclic DG-lie algebra $L$ concentrated in degrees 1 and 2, equipped with a nontrivial skew-symmetric form on $L^1$.
%\end{remark}

\begin{corollary}\label{cor.formality} 
	Let $(L,d,[-,-],(-,-))$ be a quasi-cyclic DG-Lie algebra of degree $2$ with $H^i(L)=0$ for every $i\not=0,1,2$. Assume that there exists a Lie subalgebra $H^0\subset Z^0(L)$ such that:
		\begin{enumerate} 
		
		\item the projection $Z^0(L)\to H^0(L)$ induces an isomorphism $H^0\simeq H^0(L)$;
		
		\item $L^i$ is a completely reducible $H^0$-module (with respect to the adjoint action) for all $i$;
			
\end{enumerate}
		Then the DG-Lie algebra $(L,d,[-,-])$ is formal.
		
	\end{corollary}
\begin{proof}
Construct a splitting by choosing for every $i\not=0$ a direct sum decomposition
$Z^i(L)=B^i(L)\oplus H^i$ of $H^0$-modules and then,   
for every $i$, a direct sum decomposition
$L^i=Z^i(L)\oplus K^i$ of $H^0$-modules.  This splitting satisfies the conditions of Theorem~\ref{thm.formality}.
\end{proof}

\section{Derived endomorphisms and their formality}
\label{sec.formality}

For every coherent sheaf $\sF$ on a smooth complex projective manifold  $X$ there is well-defined homotopy class of DG-Lie algebras  denoted by 
$\RHom_X(\sF,\sF)$ and called, with a little abuse of language, the \emph{DG-Lie algebra of derived endomorphisms} of $\sF$. There exists several possible (quasi-isomorphic) representatives for 
$\RHom_X(\sF,\sF)$ and we refer to \cite{Mea18} for an explicit and concrete description of many of them.
The importance of the DG-Lie algebra of derived endomorphisms relies on the facts that
it  controls the deformation theory of $\sF$ in the usual way via Maurer-Cartan equation modulus gauge action, cf.  \cite{AS,BZ18,DMcoppie,Mea18}.  
Moreover $H^i(\RHom_X(\sF,\sF))=\Ext^i(\sF,\sF)$ for every $i$.

Since the notion of quasi-cyclic DG-Lie algebra is not stable under general quasi-isomorphisms, in 
view of a possible application of Theorem~\ref{thm.formality} it is useful to consider the Dolbeault 
representatives for  $\RHom_X(\sF,\sF)$. Consider a \emph{finite locally free} 
resolution $\sE^*=\{\cdots \sE^{-1}\to \sE^0\}\to \sF$  and denote by 
\[ \HOM^*_{\Oh_X}(\sE^*,\sE^*)=\bigoplus_d\HOM^d_{\Oh_X}(\sE^*,\sE^*)=\bigoplus_d\bigoplus_{p}\HOM_{\Oh_X}\left(\sE^p,\sE^{d+p}\right)\,.\]
the (DG) sheaf of endomorphisms of $\sE^*$. Then $\HOM^*_{\Oh_X}(\sE^*,\sE^*)$ is a sheaf of 
DG-Lie algebras over $X$. It is important to notice that the bracket $[f,g]=fg-(-1)^{|f|\,|g|}gf$ 
is $\Oh_X$-bilinear  and therefore it can be extended naturally to the 
Dolbeault's resolution 
\[ L=A^{0,*}_X(\HOM^*_{\Oh_X}(\sE^*,\sE^*))=\bigoplus_{p,q,r} A^{0,p}_X(\HOM_{\Oh_X}(\sE^q,\sE^{r})),\]
where  $A^{0,p}_X(\sG)$  denotes the space of global differential forms of type $(0,q)$ with values in 
the locally free sheaf $\sG$. Similarly the usual trace map (see e.g. \cite{DMcoppie} and references therein)  
\[\Tr\colon \HOM^*_{\Oh_X}(\sE^*,\sE^*)\to \Oh_X,\qquad \Tr(f)=\sum_{i}(-1)^i\Tr(f_i^i),\text{ where } 
f=\sum_{i,j}f_{i}^{j},\quad
f_{i}^{j}\colon \sE^i\to \sE^j\,,\] 
is a morphism of sheaves of DG-Lie algebras, and  extends  to a morphism of DG-Lie algebras
\[ \Tr\colon L=A^{0,*}_X(\HOM^*_{\Oh_X}(\sE^*,\sE^*))\to  A^{0,*}_X\,.\]
Since the bracket on $A^{0,*}_X$ is trivial we have 
\[ \Tr([f,g])=0,\qquad \Tr(df)=\debar\Tr(f),\]
for every 
$f,g\in L$ and this immediately implies that $\Tr([f,g]h)=\Tr(f[h,g])$  
for every $f,g,h\in L$. If $\omega$ is a nontrivial section of the canonical bundle of $X$, the graded 
symmetric bilinear form 
\begin{equation}\label{equ.scalarproduct}
(-,-)\colon L^{\odot 2}\to \C[-\dim X],\qquad (f,g)=\int_X\omega\wedge \Tr(fg)\,,
\end{equation}
is a cyclic bilinear form, where this means that it satisfies the conditions 
$(df,g)+(-1)^{|f|}(f,dg)=0$ and $([f,g],h)=(f,[g,h])$. 
Finally if $\omega$ is a holomorphic volume form, by Serre duality the above bilinear form is non-degenerate in cohomology and therefore $(L,(-,-))$ is a quasi-cyclic DG-Lie algebra of degree $\dim X$.

\bigskip 
From now on we consider only coherent sheaves on 
projective surfaces with torsion canonical bundle.	
 	According to the Enriques-Kodaira classification of surfaces, see e.g. 
\cite{BPV,Beauville},  a smooth projective surface has torsion canonical bundle $K$ if and only if it is minimal of Kodaira dimension $0$. According to the values of irregularity $q$ and geometric genus $p_g$, these surfaces are classified into four (non empty) distinguished classes:
\begin{itemize}

\item projective K3 surfaces, with $q=0$, $p_g=1$ and $K=0$;

\item Enriques surfaces,  with $q=0$, $p_g=0$ and $2K=0$;

\item bi-elliptic surfaces, with $q=1$, $p_g=0$ and $nK=0$ for some $n=2,3,4,6$;

\item Abelian surfaces, with $q=2$, $p_g=1$ and  $K=0$.

\end{itemize}

We are now ready to prove the Kaledin-Lehn formality conjecture for the above 
surfaces, namely that $\RHom_X(\sF,\sF)$ is formal whenever $\sF$ is  polystable with respect to any (possibly non-generic) polarization, see e.g. \cite[Chapter 1]{HL}. It is useful and instructive to give first a separate proof for the cases of K3 and abelian surfaces.

\begin{theorem}\label{thm.formality1} 
Let $X$ be a complex projective surface with trivial canonical bundle and let $\sF$ be a coherent sheaf on $X$. If the group of automorphisms of $\sF$ is linearly reductive, e.g. if
$\sF$ is  polystable, then  the DG-Lie algebra 
$\RHom_X(\sF,\sF)$ is formal.
\end{theorem}

\begin{proof} Let's denote by $G$ the linearly reductive group of automorphisms of $\sF$. 
Since $X$ is smooth projective it is not difficult to see that there exists a $G$-equivariant finite locally free resolution $\sE^*=\{0\to \sE^{-2}\to \sE^{-1}\to \sE^0\}\to \sF$; a detailed proof is given for instance in \cite{quadraticity}. We claim that the DG-Lie algebra $L=A^{0,*}_X(\HOM^*_{\Oh_X}(\sE^*,\sE^*))$ satisfies the condition of Theorem~\ref{thm.formality}, when equipped with the cyclic non-degenerate structure  \eqref{equ.scalarproduct}.  

Assume for the moment that the induced action of $G$ on $L^i$ is rational for every $i$;  since the action of $G$ commutes with the differential of the resolution $\sE^*$ we have a natural inclusion $G\subset Z^0(L)$ 
and we can take $H^0=T_{\Id}G\subset Z^0(L)\subset L^0$ as the Lie algebra of $G$.  
Then, since $G$ is assumed to be linearly reductive we may extend $H^0$ to a 
$G$-equivariant splitting of $L$ that clearly satisfies the hypotheses of Theorem~\ref{thm.formality}.   
 
It remains to be shown that $L^i$ is a rational representation of $G$ for every $i$. This follows immediately from the results of \cite{quadraticity} and we give here only a sketch of proof. 
The key point  is that if $G$ acts on a coherent sheaf $\sG$ then for every open affine subset $U$, then  
the space $\sG(U)$ is a rational \emph{finitely supported} representation of $G$ \cite[Lemma 3.5]{quadraticity}. Recall that a representation is finitely supported if it is isomorphic to a finite direct sum $\bigoplus_{i=1}^nH_i\otimes W_i$, for some irreducible rational (hence finite-dimensional) representations $H_i$ and some trivial representations $W_i$; every subrepresentation and every quotient of a rational finitely supported representation remains finitely supported \cite[Lemma 2.7 and Remark 2.8]{quadraticity}.

Let $X=\bigcup_j U_j$ be a finite open affine cover such that $\sE^*$ is free over $U_j$ for every  $j$. 
Then $\Gamma\left(U_j,\HOM^*_{\Oh_X}(\sE^*,\sE^*)\right)$ is rational and finitely supported for every $j$, therefore also 
\[ L\subset \bigoplus_{j} A^{0,*}_{U_j}\left(\HOM^*_{\Oh_X}(\sE^*,\sE^*)\right)
\subset \bigoplus_{j}A^{0,*}_{U_j}\otimes_{\C} \Gamma\left(U_j, \HOM^*_{\Oh_X}(\sE^*,\sE^*)\right) \]
is a rational finitely supported representation of $G$.
\end{proof}

Let's now come back to our initial situation, namely with 
$\sF$  a  polystable sheaf  on a  smooth  projective surface $X$ with torsion canonical bundle $K_X=\Omega^2_X$. We denote by  $n$ be the smallest positive integer such that $K_X^{\otimes n}\simeq \Oh_X$ (we already know that $n=1,2,3,4,6$). 

Every choice of an isomorphism $K_X^{\otimes n}\xrightarrow{\simeq}\Oh_X$ induces naturally a structure of commutative $\Oh_X$-algebra on the locally free sheaf of rank $n$
\[ \sC:=\Oh_X\oplus K_X\oplus K^{\otimes 2}_X\oplus \cdots \oplus K^{\otimes n-1}_X\,.\]
Since $K_X$ is a torsion line bundle we have that also $\sF\otimes \sC$ is 
polystable.

Let $\sE^*=\{\cdots \sE^{-1}\to \sE^0\}\to \sF$ be any finite locally free resolution, 
then $\sE^*\otimes\sC$ is a finite locally free resolution of $\sF\otimes \sC$.
Moreover 
\[ \HOM^*_{\Oh_X}(\sE^*\otimes\sC,\sE^*\otimes\sC)=\bigoplus_{i,j=0}^{n-1}
\HOM^*_{\Oh_X}\left(\sE^*\otimes K_X^{\otimes i},\sE^*\otimes K_X^{\otimes j}\right),\]
and every direct summand is a $\HOM^*_{\Oh_X}(\sE^*,\sE^*)$-module via the adjoint action. 
The trace map extends  naturally to a morphism of sheaves 
\[\begin{split}
\widetilde{\Tr}\colon &\HOM^*_{\Oh_X}(\sE^*\otimes\sC,\sE^*\otimes\sC)\to \sC,\\[3pt] 
&\text{with components}\\
\widetilde{\Tr}\colon &\HOM^*_{\Oh_X}\left(\sE^*\otimes K_X^{\otimes i},\sE^*\otimes K_X^{\otimes j}\right)=
\HOM^*_{\Oh_X}\left(\sE^*,\sE^*\right)\otimes K_X^{\otimes j-i}\xrightarrow{\Tr\otimes\Id} K_X^{\otimes j-i}\,.
\end{split}\]
The DG-Lie algebra $L=A^{0,*}_X(\HOM^*_{\Oh_X}(\sE^*\otimes \sC,\sE^*\otimes \sC))$
is quasi-cyclic of degree 2, when equipped with the pairing  
\[ (f,g)=\int_X p_{K}\widetilde{\Tr}(fg)\,,\]
where
$p_K\colon A^{0,*}_X(\sC)\to A^{0,2}_X(K_X)=A^{2,2}_X$ is the projection. In fact, for every $0\le i,j<n$ the above pairing induces the Serre duality isomorphism 
\[\Ext^h_X\left(\sF\otimes K^{\otimes i}_X, \sF\otimes K^{\otimes j}_X\right)\simeq 
\Ext^{2-h}_X\left(\sF\otimes K^{\otimes j}_X, \sF\otimes K^{\otimes i+1}_X\right)^\vee \]
so that it is non-degenerate in cohomology.

\begin{lemma}\label{lem.risoluzioneadattata} 
In the above situation there exists a 
finite locally free resolution $\sE^*\to \sF$ such that every endomorphism of 
$\sF\otimes \sC$ lifts canonically to an endomorphism of the complex 
$\sE^*\otimes \sC$.
\end{lemma}

\begin{proof} 
By assumption $\sF$ is a pure coherent sheaf that is a direct sum of stable sheaves 
with the same reduced Hilbert polynomial:
\[ \sF=\sF_1\oplus \cdots\oplus \sF_n\,.\]
In particular $\Hom_{\Oh_X}(\sF_i,\sF_j)=0$ for every $i\not=j$ and 
$\Hom_{\Oh_X}(\sF_i,\sF_i)=\C$ for every $i$.
Consider the following equivalence relation on the set of direct summands 
\[ \sF_i\sim \sF_j\iff \sF_i\otimes \sC\cong \sF_j\otimes \sC\,.\] 
Equivalently  $\sF_i\sim \sF_j$ if and only if $\sF_i$ is isomorphic to 
$\sF_j\otimes K^{\otimes h}$ for some $h$. Up to permutation of indices we may assume that 
$\sF_1,\ldots,\sF_r$ are a set or representatives for this equivalence relation. We may write
\[ \sF=\bigoplus_{i=1}^r \sF_i\otimes \sW_i\,,\]
where every $\sW_i$ is a direct sum of line bundles of type 
$K_X^{\otimes h}$. We have  
$\sF\otimes\sC=\bigoplus_{i=1}^r \sF_i\otimes \sC^{\oplus w_i}$ where $w_i$ is the rank of 
$\sW_i$. Every non-trivial endomorphism of $\sF_i$ is a scalar multiple of the identity and then
the group of automorphisms of $\sF\otimes\sC$ is the product of $n$ copies of 
$\prod_{i=1}^rGL_{w_i}(\C)$.

Choose $r$ finite locally free resolution 
$\sE_i^*\to \sF_i$: every endomorphism of $\sF_i$ is a scalar multiple of the identity and then  lifts canonically to $\sE_i^*$. It is now easy to verify that 
\[ \sE^*=\bigoplus_{i=1}^r \sE^*_j\otimes \sW_i\]
is resolution of $\sF$ with the required properties.
\end{proof}

\begin{theorem} \label{thm.formality2}
In the above situation both the 
DG-Lie algebras  $\RHom_X(\sF\otimes\sC,\sF\otimes\sC)$ and $\RHom_X(\sF,\sF)$ are formal.
\end{theorem}

\begin{proof}
Let $\sE\to \sF$ be a resolution as in Lemma~\ref{lem.risoluzioneadattata} and consider the quasi-cyclic DG-Lie algebra 
\[ L=A^{0,*}_X(\HOM^*_{\Oh_X}(\sE^*\otimes \sC,\sE^*\otimes \sC))\]
as a representative in the homotopy class of $R\Hom(\sF\otimes\sC,\sF\otimes\sC)$. The same arguments used in the proof of  Theorem~\ref{thm.formality1} imply that $L$ is a rational and finitely supported representation of the linearly reductive group of automorphisms of $\sF\otimes\sC$.

By assumption there exists a natural inclusion of Lie algebras 
\[ \Hom_X(\sF\otimes\sC,\sF\otimes\sC)\simeq H^0\subset \Hom_X(\sE^*\otimes\sC,\sE^*\otimes\sC)=Z^0(L)\]
that induces an isomorphism  $H^0\simeq H^0(L)$.
The adjoint action of $\Hom_X(\sF\otimes\sC,\sF\otimes\sC)$ on $L$ is induced by a rational action 
of a linearly reductive algebraic group, hence the action of $H^0$ on $L$ is  completely reducible and the formality of $L$ follows from Corollary~\ref{cor.formality}.

Taking 
\[ M=A^{0,*}_X(\HOM^*_{\Oh_X}(\sE^*,\sE^*))\]
as a representative in the homotopy class of $R\Hom(\sF,\sF)$, we have already observed that 
there exists a natural inclusion of DG-Lie algebra $M\subset L$ together a decomposition of 
$L$ as a direct sum of $M$-modules:
\[ L=\bigoplus_{i,j=0}^{n-1}
A_X^{0,*}\left(\HOM^*_{\Oh_X}\left(\sE^*\otimes K_X^{\otimes i},\sE^*\otimes K_X^{\otimes j}\right)\right)\,.\]

Now the formality of $M$ is a direct consequence of the formality of $L$ and of the formality transfer theorem \ref{thm.formalitytransfer}.
\end{proof}

\noindent\textbf{Acknowledgements.} 
This work is carried out in the framework of the PRIN project  \lq\lq Moduli  and Lie theory\rq\rq\   2017YRA3LK.

\end{document}